\pdfoutput=1
\documentclass[11pt, reqno]{amsart}
\numberwithin{equation}{section}
\usepackage{amsmath,amsfonts,amssymb}
\usepackage{mathtools}
\usepackage[utf8]{inputenc}
\usepackage[english]{babel}
\usepackage{tikz-cd}
\usepackage{bbm}
\usepackage{siunitx}
\usepackage{lineno}
\usepackage{xcolor}
\usepackage{comment}
\usepackage{setspace}
\usepackage{float}

\usepackage[table]{xcolor}

\allowdisplaybreaks

\usepackage{algorithm}
\usepackage{algpseudocode}
\usepackage{subcaption}
\usepackage{graphicx}
\usepackage{tablefootnote} 
\usepackage{footnotehyper} 
\usepackage[colorlinks=true,linkcolor=black,anchorcolor=black,
citecolor=black,filecolor=black,menucolor=black,runcolor=black,
urlcolor=black]{hyperref}
\usepackage[margin=1in,footskip=0.25in]{geometry}
\usepackage{enumitem}

\usepackage{booktabs}
\usepackage{multirow}
\usepackage{pgfplots}
\pgfplotsset{compat=1.18}
\usetikzlibrary{intersections}
\usepgfplotslibrary{fillbetween}

\setlist[enumerate]{itemsep=0.4em}

\definecolor{surgeryred}{RGB}{215,48,39}
\definecolor{recoverygray}{RGB}{240,240,240}
\definecolor{chemoblue}{RGB}{66,146,198}
\definecolor{radiationorange}{RGB}{230,145,56}
\definecolor{tmzpurple}{RGB}{118,42,131}

\definecolor{continuous}{RGB}{158,188,218}
\definecolor{intermittent}{RGB}{218,158,188}

\DeclareUnicodeCharacter{FB01}{fi}
\newtheorem{theorem}{Theorem}[section]
\newtheorem{corollary}{Corollary}[theorem]

\newtheorem{prop}{Proposition}
\newtheorem*{aim-non}{Aim}

\theoremstyle{definition}
\newtheorem{definition}{Definition}[section]
\newtheorem{remark}[definition]{Remark}

\begin{document}

	\title[Mathematical Model to Predict Growth and Treatment for UPS Cancer]{A Mathematical Model to Predict Growth and Treatment for UPS Cancer}

	\author[S. Roy]{Sumit Roy}
\address{Stat-Math Unit, Indian Statistical Institute, 203 B.T. Road, Kolkata 700 108, India.}
 \email{sumitroy\_r@isical.ac.in}
\thanks{Address: Stat-Math Unit, Indian Statistical Institute, 203 B.T. Road, Kolkata 700 108, India.}

\begin{abstract}
We propose a mathematical model for the growth and treatment of Undifferentiated Pleomorphic Sarcoma (UPS) using a system of nonlinear differential equations. The model combines Gompertz-type tumor growth with surface-dependent necrotic loss, surgical resection with residual disease, postoperative recovery, tumor--immune interaction, and radiation treatment scheduling. We study the mathematical properties of the model and obtain several results. The growth equation shows the existence of a threshold below which the tumor cannot survive and may disappear. The postoperative phase exhibits an early inflammatory stage followed by proliferative recovery. For the tumor-immune subsystem, equilibrium states and local stability conditions are identified. The radiation treatment problem is formulated as an optimal control problem, and the optimal strategy is shown to be of bang-bang type. The model suggests that tumor recurrence depends not only on tumor growth itself but also on residual disease, postoperative dynamics, immune response, and treatment timing.
\end{abstract}

\dedicatory{Dedicated to my brother}
\maketitle

\section{Introduction}

This paper presents a mathematical model for studying the growth and
treatment dynamics of \textit{Undifferentiated Pleomorphic Sarcoma}
(UPS) (see \cite{Winchester2018}). We describe the basic growth of the tumor using a differential equation. This equation combines a standard Gompertz growth model (see \cite{Laird1964}, \cite{Wheldon1988}) with a term that accounts for cell death (necrosis) based on the tumor's surface area
\begin{equation*}
\frac{dV}{dt} = f(V) = r_{g}V \ln\left(\frac{K}{V}\right) - \lambda V^{2/3}. 
\end{equation*}
In this formula, $V(t)$ is the tumor volume at time $t$. $r_g$ is how fast the cells multiply, $K$ is the maximum size the body can support, and $\lambda$ is the rate at which cells die due to lack of nutrients. We prove that if the tumor becomes too small, the death rate becomes higher than the growth rate, and the tumor disappears completely in a short time. Results of this type illustrate how relatively simple growth models can still yield useful mathematical insight into tumor dynamics.

In recent years, mathematical oncology has moved beyond the use of classical growth laws alone and increasingly focused on models that incorporate treatment response, immune effects, and tumor microenvironmental factors. For example, Yin et al.~\cite{Yin2019}
surveyed mathematical models for tumor dynamics and treatment resistance in solid tumors, emphasizing the role of resistance
mechanisms in therapy response. Rockne et al.~\cite{Rockne2019Roadmap} outlined major directions in mathematical oncology, highlighting the growing role of mechanistic and predictive models in cancer research. More recently, Butner et al.~\cite{Butner2022} discussed mathematical models of cancer immunotherapy and their potential for clinically
translatable and personalized treatment strategies. Hybrid and multiscale approaches have also been developed to combine different
aspects of tumor growth and treatment within unified computational frameworks \cite{Chamseddine2020,SinghPaquin2024}. At the same time, recent studies in radiotherapy modeling emphasize that treatment outcome depends strongly on the underlying mathematical
description of radiation response and dose scheduling \cite{Kutuva2023,Zheng2025}. On the clinical side, recent reviews of
Undifferentiated Pleomorphic Sarcoma and myxofibrosarcoma have highlighted the importance of surgery, recurrence risk, radiotherapy, and emerging immune-based treatment strategies \cite{Crago2022,Dalal2024, Sun2023}.

These developments motivate the need for mathematically tractable models that connect tumor growth with surgery, postoperative recovery, treatment timing, and host response. However, much of the existing literature studies growth, immune interaction, radiotherapy, or treatment optimization separately. In the present work, we study a single analytical framework for UPS that combines intrinsic tumor growth, surgical resection, postoperative inflammatory recovery, treatment timing, and tumor-immune interaction. Here we construct a model that captures several key stages of tumor progression and treatment. Our model is divided into four main parts
\begin{enumerate}
    \item[(i)] \textit{Surgery Model}: If $V$ is the volume before surgery, the volume after surgery is $\mathcal{R}(V) = (1-\eta)V + \epsilon$. Here, $\eta$ is the fraction of tumor volume removed during surgery, and $\epsilon$ represents residual microscopic disease.
    
    \item[(ii)] \textit{Two-Phase Growth}: After surgery, the tumor goes through two steps. First is the \textit{Inflammatory Phase}. We found a specific formula to solve this
    \[
    V(t) = \left[\left(V_{res}^{1/3} - \frac{\kappa}{r}\right)e^{rt/3} + \frac{\kappa}{r}\right]^3.
    \]
    The second step is the \textit{Proliferative Phase}. 
    
    \item[(iii)] \textit{The Switch}: To move from the first phase to the second, we use a switching function $\phi(t)$. This function uses a chemical marker $[c]$ in the body that follows the rule $\frac{dc}{dt} = \beta - \mu c$. This ensures the transition in our model is smooth and realistic.
    
    \item[(iv)] \textit{Radiation Timing}: We use \textit{Pontryagin’s Minimum Principle} to find the best time to give radiation. Since the tumor's sensitivity to radiation changes over time ($\Psi(t)$), the best strategy is a bang-bang control (see \cite{Hall2018} for details). This means giving the maximum dose ($D_{max}$) when the tumor is most sensitive
    \[
    D^*(t) = \begin{cases} D_{max} & \text{if } \Psi(t) > \text{threshold} \\ 0 & \text{if } \Psi(t) < \text{threshold} .\end{cases}
    \]
\end{enumerate}

In addition to these, we also look at the interaction between the tumor ($V$) and the immune system ($E$) using the system
\begin{align*}
\frac{dV}{dt} &= rV \ln(K/V) - \frac{\delta E V}{m + V} \\
\frac{dE}{dt} &= s + \frac{\rho E V^2}{\eta^2 + V^2} - \mu E.
\end{align*}
We use the \textit{Routh-Hurwitz criterion} to prove when the immune system is strong enough to keep the tumor at a stable, small size.

Beyond the formulation of the model, the paper provides several analytical results concerning the resulting dynamical system. In particular, we identify conditions under which the necrosis term induces a minimum viable tumor size, derive explicit solutions for the inflammatory growth phase, and analyze the stability of the tumor-immune equilibrium using the Routh-Hurwitz criterion. Furthermore, we prove that the optimal radiation schedule is of
bang--bang type, showing that radiation should be administered at the maximum allowable rate during periods of high radiosensitivity (Theorem~\ref{thm:rt}). The present framework therefore provides a unified dynamical systems model for UPS that permits the mathematical analysis of recurrence thresholds, postoperative phase transitions, and treatment timing.

Finally, we compare the qualitative predictions of the model with clinical observations reported for UPS. In the future, we plan to add stochastic equations to the model to account for how every patient is slightly different.

Throughout the paper, all model parameters are assumed to be real and nonnegative unless otherwise stated. Their biological meaning is summarized in Table~\ref{tab:param_meaning}.
\begin{table}[H]
\centering
\caption{Model parameters and biological interpretation}
\label{tab:param_meaning}
\begin{tabular}{lll}
\toprule
Parameter & Mathematical assumption & Biological meaning\\
\midrule
$r_g$ & $>0$ & intrinsic tumor proliferation rate\\
$K$ & $>0$ & carrying capacity\\
$\lambda$ & $>0$ & necrotic loss coefficient\\
$\eta$ & $\in(0,1)$ & surgical resection efficiency\\
$\epsilon$ & $>0$ & residual tumor volume after surgery\\
$r_p$ & $>0$ & postoperative proliferation rate\\
$r_{\mathrm{inflam}}$ & $\ge0$ & inflammation-driven growth rate\\
$\delta_{\mathrm{immune}}$ & $\ge0$ & immune-mediated tumor clearance rate\\
$\kappa_{\mathrm{hypoxia}}$ & $\ge0$ & hypoxia-induced cell loss coefficient\\
$\alpha,\beta$ & $\ge0$ & radiosensitivity parameters\\
$\tau$ & $>0$ & radiosensitivity period\\
$D(t)$ & $\ge0$ & radiation dose rate\\
$D_{\max}$ & $>0$ & maximum allowable dose rate\\
$D_{\mathrm{total}}$ & $>0$ & total radiation budget\\
$\delta$ & $\ge0$ & adjuvant treatment kill rate\\
$s$ & $\ge0$ & baseline immune influx\\
$\rho$ & $\ge0$ & immune proliferation rate\\
$\mu$ & $>0$ & immune cell decay rate\\
$m$ & $>0$ & immune saturation constant\\
\bottomrule
\end{tabular}
\end{table}

\section{Tumor Growth Dynamics}
\label{sec:growth_model}
Consider the following tumor growth model where proliferation follows Gompertz kinetics and cell loss scales with surface area, i.e.

\begin{equation}
\frac{dV}{dt} = f(V)= \underbrace{r_g V \ln\left(\frac{K}{V}\right)}_{\text{proliferation}} - \underbrace{\lambda V^{2/3}}_{\text{necrosis}},
\label{eq:growth}
\end{equation}
where $V(t)$ represents tumor volume at time $t$, $r_g$ is the proliferation rate, $K$ is the carrying capacity, and $\lambda$ is the necrosis coefficient.

\begin{theorem}
Consider the tumor growth model \eqref{eq:growth} with $\lambda > 0$. 
\begin{enumerate}
    \item[(i)] 
    If the tumor volume $V(t)$ enters the regime where $$r_g V \ln\left(\frac{K}{V}\right) < \lambda V^{2/3},$$ the trajectory $V(t)$ terminates at the boundary $V=0$ in finite time $t_{ext} < \infty$.
    \vspace{0.2cm}
    
    \item[(ii)] 
There exists $\delta>0$ such that, for every solution with initial
value $V(0)\in(0,K)$,
\[
V(t)\le K-\delta
\]
for all sufficiently large $t$. In particular,
\[
\limsup_{t\to\infty}V(t)\le K-\delta<K.
\]
\end{enumerate}
\end{theorem}

\begin{proof}
$(i)$ Let $$g(V) = r_g V^{1/3} \ln\left(\frac{K}{V}\right) - \lambda.$$ Since $\lim_{V \to 0^+} g(V) = -\lambda$, there exists a $\delta > 0$ such that for all $V \in (0, \delta)$, $$\frac{dV}{dt} = V^{2/3} g(V) \leq -\frac{\lambda}{2} V^{2/3}.$$ Now
\[
\frac{d}{dt}\bigl(V(t)^{1/3}\bigr) = \frac{1}{3}V(t)^{-2/3}\frac{dV}{dt} \le - \frac{\lambda}{6}.
\]
Hence, integrating from $0$ to $t$, we obtain
\[
V(t)^{1/3} \le V(0)^{1/3} - \frac{\lambda t}{6}.
\]
Thus, $V(t)$ must reach zero at some time $t_{ext} \leq \frac{6V(0)^{1/3}}{\lambda}$. 
\vspace{0.2cm}

$(ii)$ At $V=K$, we have
\[
f(K)=r_gK\ln(1)-\lambda K^{2/3}
=-\lambda K^{2/3}<0.
\]
By continuity, there exists $\delta>0$ such that
\[
f(V)<0
\qquad \text{for all }V\in[K-\delta,K].
\]
Since $[K-\delta,K]$ is compact, there exists a constant $c>0$ such that
\[
f(V)\le -c
\qquad \text{for all }V\in[K-\delta,K].
\]

Now suppose that a trajectory enters $[K-\delta,K]$. While it remains
in this interval, we have
\[
\frac{dV}{dt}=f(V)\le -c.
\]
Therefore the trajectory must leave the interval through the lower
endpoint $K-\delta$ in finite time. Once it has moved below
$K-\delta$, it cannot cross this level upward again, since
\[
f(K-\delta)<0.
\]
Thus every solution with initial value in $(0,K)$ eventually remains
below $K-\delta$. Consequently,
\[
\limsup_{t\to\infty}V(t)\le K-\delta<K.
\]
\end{proof}

\begin{remark}
Part $(i)$ of the above theorem represents the Minimum Viable Volume. In the present model, this refers to a threshold region below which the necrotic loss term dominates the proliferative growth term, causing the tumor volume to decrease toward extinction. Unlike the standard Gompertz model, where arbitrarily small tumor volumes may still grow, the present model suggests that sufficiently small tumors may fail to sustain growth and instead undergo clearance. Thus, the minimum viable volume may be viewed as a qualitative boundary between extinction and sustained tumor growth.
\end{remark}

Before studying stability, we briefly explain the role of linearization. If $V_\infty$ is an equilibrium of the nonlinear model
\[
\frac{dV}{dt}=f(V),
\]
then the behavior of nearby solutions can be understood by studying the linearized system around $V_\infty$. Local asymptotic stability means that solutions starting sufficiently close to $V_\infty$ approach this equilibrium as $t\to\infty$. Exponential convergence means that the distance to equilibrium decays at an exponential rate, i.e.,
\[
|V(t)-V_\infty|
\le
Ce^{-\alpha t}
\]
for some constants $C,\alpha>0$. The use of linearization near an equilibrium is standard and is supported by results such as the Hartman-Grobman theorem.

\begin{theorem}\label{thm_preop}
Let us consider the tumor growth model \eqref{eq:growth} and let $V_{\infty}$ be the non-trivial stable equilibrium\footnote{Stability of $V_{\infty}$ and the positivity of the convergence rate $\alpha$ are guaranteed provided the steady-state volume satisfies $V_{\infty} > Ke^{-3}$, ensuring $1 - \frac{1}{3}\ln(K/V_{\infty}) > 0$.}. Then the exponential convergence rate $\alpha$ for the linearized system is given by
\begin{equation}
\alpha = r_{g} - \frac{1}{3} \lambda V_{\infty}^{-1/3} = r_{g} \left[ 1 - \frac{1}{3} \ln \left( \frac{K}{V_{\infty}} \right) \right].
\end{equation}
\end{theorem}

\begin{proof}
The equilibrium $V_{\infty}$ is defined by the condition $f(V_{\infty}) = 0$, and for $V_{\infty} > 0$, this means
\begin{equation}\label{eq:eq_cond}
r_{g} \ln\left(\frac{K}{V_{\infty}}\right) = \lambda V_{\infty}^{-1/3}.
\end{equation}

We will linearize $f(V)$ to evaluate the stability. Using the product rule on the Gompertz term, we get
\begin{align*}
f'(V) &= \frac{d}{dV} \left[ r_{g} V (\ln K - \ln V) - \lambda V^{2/3} \right] \\
&= r_{g} \left[ (\ln K - \ln V) + V \left( -\frac{1}{V} \right) \right] - \frac{2}{3} \lambda V^{-1/3} \\
&= r_{g} \ln \left( \frac{K}{V} \right) - r_{g} - \frac{2}{3} \lambda V^{-1/3}.
\end{align*}

Therefore at the equilibrium point $V = V_{\infty}$, we get
\begin{equation*}
f'(V_{\infty}) = r_{g} \ln \left( \frac{K}{V_{\infty}} \right) - r_{g} - \frac{2}{3} \lambda V_{\infty}^{-1/3}.
\end{equation*}

Thus by substituting the equilibrium identity from \eqref{eq:eq_cond}, we obtain
\begin{align*}
f'(V_{\infty}) &= \left( \lambda V_{\infty}^{-1/3} \right) - r_{g} - \frac{2}{3} \lambda V_{\infty}^{-1/3} \\
&= \frac{1}{3} \lambda V_{\infty}^{-1/3} - r_{g}.
\end{align*}

To identify the convergence rate, we consider a small perturbation $x=V-V_\infty$. The linearized equation near $V_\infty$ is
\[
\dot x=f'(V_\infty)x.
\]
Hence
\[
x(t)\sim e^{f'(V_\infty)t}.
\]
For a stable equilibrium, we therefore define the exponential convergence rate by
\[
\alpha=-f'(V_\infty)>0,
\]
so that perturbations decay like
\[
x(t)\sim e^{-\alpha t}.
\]
Thus
\begin{equation}
\alpha = -f'(V_\infty)= r_{g} - \frac{1}{3} \lambda V_{\infty}^{-1/3}.
\end{equation}
Again substituting \eqref{eq:eq_cond}, we can also express this by
\begin{equation}
\alpha = r_{g} \left[ 1 - \frac{1}{3} \ln \left( \frac{K}{V_{\infty}} \right) \right]
\end{equation}
\end{proof}

\begin{remark}\label{rem:biology}
The model components describe distinct tumor growth dynamics: the proliferation term $r_g V \ln(K/V)$ captures the transition from quasi-exponential expansion to decelerated growth as the tumor volume $V$ approaches the carrying capacity $K$, reflecting resource depletion and spatial constraints. The inhibitory term $-\lambda V^{2/3}$ models surface-area dependent effects; specifically, it accounts for nutrient diffusion limitations where the viable proliferating fraction is restricted to the tumor's outer shell, eventually leading to hypoxic core formation in tumors larger than 1 cm$^3$. Table~\ref{tab:params} provides characteristic parameters for various malignancies, identifying glioblastoma as having the highest proliferation rate (0.143 day$^{-1}$), whereas UPS demonstrates aggressive metastatic potential despite intermediate primary growth rates \cite{Benzekry2014,Eilber2004,Lambin2012,Norton1988,Swanson2008}.
\end{remark}

\begin{figure}[h!]
\centering
\includegraphics[width=0.8\linewidth]{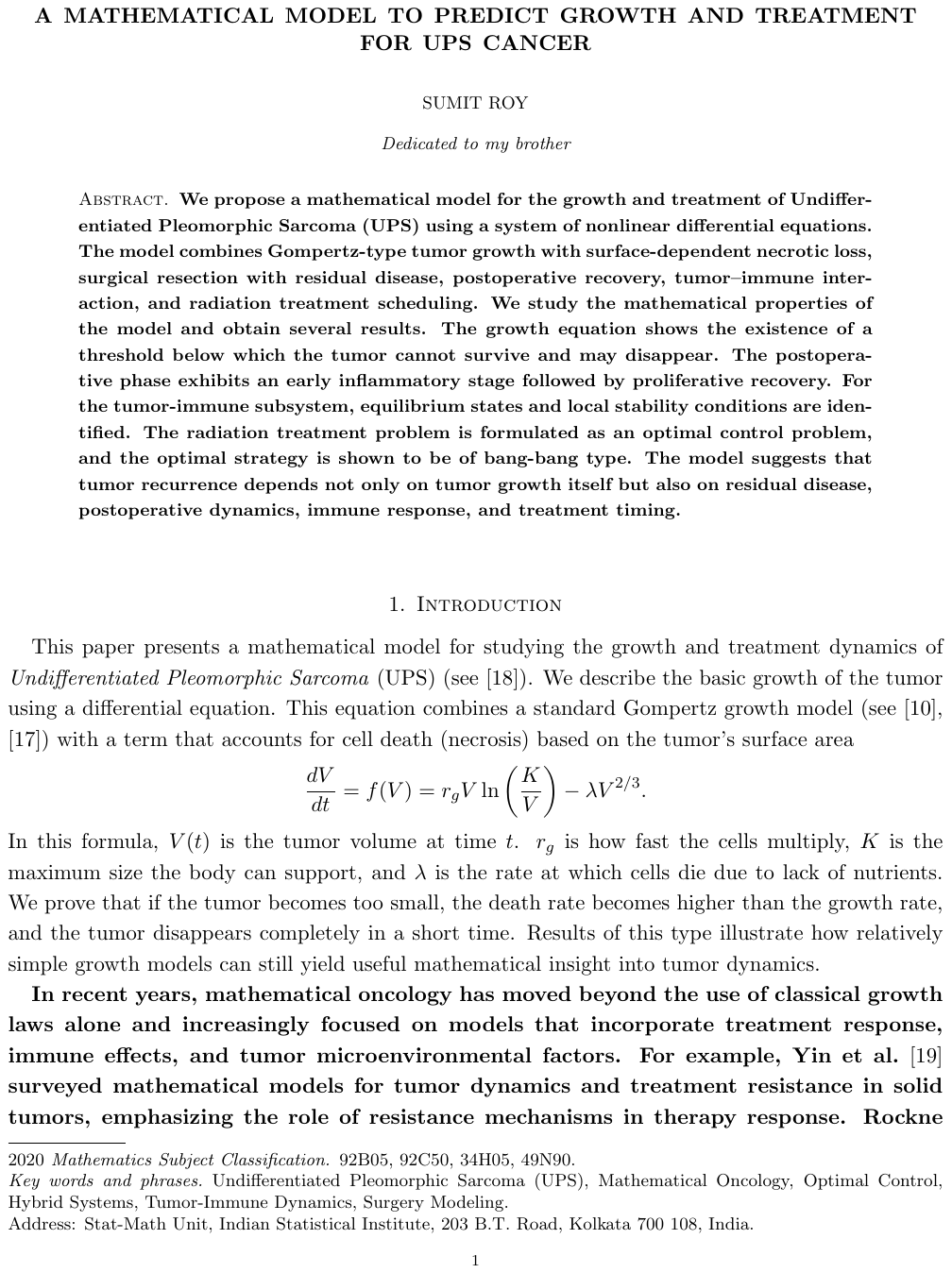}
\caption{Comparison of tumor growth trajectories predicted by the modified Gompertz--necrosis model for representative tumor types.}
\label{fig:growth}
\end{figure}

Figure~\ref{fig:growth} illustrates the tumor growth trajectories predicted by the modified Gompertz--necrosis model for several tumor types. The trajectories exhibit the typical sigmoidal behavior associated with Gompertzian growth. For small tumor volumes the proliferative term $r_g V \ln(K/V)$ dominates the necrotic loss term $\lambda V^{2/3}$, producing rapid early expansion. As the tumor volume increases, however, the logarithmic slowdown in the Gompertz component together with the surface-mediated loss term gradually reduces the net growth rate, causing the curves to bend and approach a slower long-term growth regime.

Differences between the trajectories reflect the biological variability between tumor types. In particular, tumors with larger effective proliferation rates exhibit steeper early growth, while those with stronger inhibitory effects display more gradual progression. The UPS trajectory lies among the more aggressive profiles in the simulation over the time interval considered, which is consistent with the clinically reported rapid progression of this tumor type.

\begin{table}[H]
\centering
\caption{Tumor growth parameters}\label{tab:params}
\begin{tabular}{lccc}
\toprule
Tumor Type & $r_g$ (day$^{-1}$) & $\lambda$ (day$^{-1}$) & $K$ (cm$^3$) \\
\midrule
Breast & 0.072 & 0.018 & 1200 \\
Glioblastoma & 0.143 & 0.027 & 800 \\
Prostate & 0.032 & 0.008 & 1500 \\
Sarcoma & 0.095 & 0.022 & 2500 \\
UPS & 0.108 & 0.026 & 1800 \\
\hline
\end{tabular}
\end{table}

Table~\ref{tab:params} lists representative parameter values used in the growth simulations, following standard Gompertz-type tumor growth models (see, for example, \cite{Benzekry2014, Laird1964,Norton1988}). The parameter $r_g$ represents the intrinsic proliferation rate of tumor cells, while $\lambda$ models cell loss due to necrosis and nutrient limitations near the tumor surface. The carrying capacity $K$ reflects environmental constraints on the maximum tumor volume that can be sustained. Variations in these parameters allow the model to reproduce the different growth patterns observed across tumor
types while preserving the same underlying dynamical structure.

\section{Surgical Resection Module}

\subsection{Mathematical Formulation and Analysis}

Let $\Omega = \{V \in \mathbb{R} : 0 < V < K\}$ denote the biologically relevant range of positive tumor volumes. In the postoperative analysis below, we restrict attention to trajectories that remain in the positive regime and evolve in a neighborhood of the positive postoperative equilibrium.

We denote by $t_{op}^+$ the time immediately after surgery, and write
\[
V(t_{op}^+)=V_{\mathrm{res}}
\]
for the residual tumor volume following surgical resection.

\begin{theorem}
\label{thm:postop_dynamics}
Let the postoperative tumor volume $V(t)$ be governed by the non-linear system
\begin{equation}
\frac{dV}{dt} = f_p(V) = r_p V \ln\left(\frac{K_{\text{post}}}{V}\right) - \lambda V^{2/3}, \quad V(t_{op}^+) = V_{\text{res}}
\end{equation}
where $r_p, K_{\text{post}}, \lambda > 0$. Assume that the postoperative system admits a positive equilibrium
$V_{\infty,p}\in(0,K_{\text{post}})$ satisfying
\[
r_p\ln\left(\frac{K_{\text{post}}}{V_{\infty,p}}\right)
=\lambda V_{\infty,p}^{-1/3}.
\]
The dynamics satisfy
\begin{enumerate}
    \item[(i)] 
    The equilibrium $V_{\infty, p}$ is locally asymptotically stable provided $V_{\infty, p} > K_{\text{post}}e^{-3}$.
    \vspace{0.15cm}
    
    \item[(ii)] 
    The local convergence rate $\alpha_p$ is given by:
    \begin{equation}
    \alpha_p = r_p \left( 1 - \frac{1}{3} \ln \left( \frac{K_{\text{post}}}{V_{\infty, p}} \right) \right).
    \end{equation}
  \item[(iii)] 
   For trajectories starting sufficiently close to $V_{\infty,p}$, the
linearized dynamics satisfy
\[
|x(t)| = |x(t_{op})| e^{-\alpha_p (t-t_{op})},
\qquad x(t)=V(t)-V_{\infty,p}.
\]
In particular, the nonlinear system exhibits local exponential decay toward $V_{\infty,p}$.
\end{enumerate}
\end{theorem}
\begin{proof}
For the local stability, we linearize $f_p(V)$ about the steady state $V_{\infty, p}$. Let $x(t) = V(t) - V_{\infty, p}$ denote a small perturbation. Then the first-order expansion gives $\frac{dx}{dt} = f_p'(V_{\infty, p})x$. Therefore
\begin{align*}
f_p'(V) &= \frac{d}{dV} \left[ r_p V (\ln K_{\text{post}} - \ln V) - \lambda V^{2/3} \right] \\
&= r_p \ln \left( \frac{K_{\text{post}}}{V} \right) - r_p - \frac{2}{3} \lambda V^{-1/3}
\end{align*}
At the steady state $V_{\infty, p}$, by substitute the equilibrium identity $\lambda V_{\infty, p}^{-1/3} = r_p \ln\left(\frac{K_{\text{post}}}{V_{\infty, p}}\right)$, we obtain
\begin{align*}
f_p'(V_{\infty, p}) &= r_p \ln\left(\frac{K_{\text{post}}}{V_{\infty, p}}\right) - r_p - \frac{2}{3} \left( r_p \ln\left(\frac{K_{\text{post}}}{V_{\infty, p}}\right) \right) \\
&= \frac{1}{3} r_p \ln\left(\frac{K_{\text{post}}}{V_{\infty, p}}\right) - r_p \\
&= -r_p \left( 1 - \frac{1}{3} \ln \left( \frac{K_{\text{post}}}{V_{\infty, p}} \right) \right)
\end{align*}
By defining $\alpha_p = -f_p'(V_{\infty, p})$, we observe that $\alpha_p > 0$ for $V_{\infty, p} > K_{\text{post}}e^{-3}$. According to the Hartman-Grobman Theorem, the linearized dynamics $\frac{dx}{dt} = -\alpha_p x$ characterize the local stability. This gives $(i)$ and $(ii)$. 

For part $(iii)$, let
\[
x(t)=V(t)-V_{\infty,p}
\]
denote the deviation from the equilibrium. Since $V_{\infty,p}$ is an equilibrium of the postoperative growth equation
\[
\frac{dV}{dt}=f_p(V),
\]
we have
\[
f_p(V_{\infty,p})=0.
\]
Expanding $f_p(V)$ about $V_{\infty,p}$, we obtain
\[
f_p(V)=f_p(V_{\infty,p})+f_p'(V_{\infty,p})(V-V_{\infty,p})+\mathcal O\!\left((V-V_{\infty,p})^2\right).
\]
Therefore
\[
\dot{x}(t)=f_p'(V_{\infty,p})\,x(t)+\mathcal O(x(t)^2).
\]
Neglecting the higher-order term gives the linearized equation
\[
\dot{x}(t)=f_p'(V_{\infty,p})\,x(t).
\]
From part (ii), we have
\[
\alpha_p=-f_p'(V_{\infty,p})>0,
\]
so the linearized system may be written as
\[
\dot{x}(t)=-\alpha_p x(t).
\]
This is a first-order linear equation, and its solution with initial
condition $x(t_{op})=V_{\mathrm{res}}-V_{\infty,p}$ is
\[
x(t)=x(t_{op})e^{-\alpha_p(t-t_{op})}.
\]
Hence
\[
|x(t)|=|x(t_{op})|e^{-\alpha_p(t-t_{op})}.
\]

This proves the exponential decay for the linearized dynamics. Since $f_p$ is continuously differentiable near $V_{\infty,p}$ and $f_p'(V_{\infty,p})<0$, standard linearization theory implies that the equilibrium $V_{\infty,p}$ is locally exponentially stable for the nonlinear system as well. Thus the nonlinear system inherits local exponential convergence toward $V_{\infty,p}$, from its linearization.
\end{proof}

We model surgical intervention through an affine reduction operator, assuming that surgery removes a fixed proportion of the tumor while a small residual burden may remain after resection.

\begin{theorem}\label{Res_op}
Let $\mathcal{R}: \mathbb{R}^+ \to \mathbb{R}^+$ be the surgical operator defined by $$\mathcal{R}(V) = (1-\eta)V + \epsilon,$$ where $\eta \in (0,1)$ is the resection efficiency and $\epsilon > 0$ is the minimum residual volume. Then
\begin{enumerate}
    \item[(i)] 
    $\mathcal{R}$ has a unique fixed point at $V^* = \epsilon/\eta$, representing the volume invariant under the operator's action.
    \vspace{0.15cm}
    
    \item[(ii)] 
    $\mathcal{R}$ is a global contraction on $(0, \infty)$ with Lipschitz constant $L = 1-\eta$.
        \vspace{0.15cm}

    \item[(iii)] 
    For any $V > 0$, $\mathcal{R}(V) \geq \epsilon$, ensuring the presence of residual disease post-intervention.
\end{enumerate}
\end{theorem}

\begin{proof}
$(i)$ Suppose $V^*$ is a fixed point of $\mathcal{R}$, i.e. $\mathcal{R}(V^*) = V^*$. Then
\begin{align*}
     (1-\eta)V^* + \epsilon &= V^*\\
   \Rightarrow \eta V^* &= \epsilon\\
   \Rightarrow V^* &= \epsilon/\eta.
\end{align*}
Since $\eta \neq 0$, the point $V^* = \epsilon/\eta$ is unique.
\vspace{0.15cm}

$(ii)$ For any $V_1, V_2 \in \mathbb{R}^+$, the Euclidean distance between images $\mathcal{R}(V_1)$ and $\mathcal{R}(V_2)$ is given by
\begin{align*}
    d(\mathcal{R}(V_1), \mathcal{R}(V_2)) &= |(1-\eta)V_1 + \epsilon - (1-\eta)V_2 - \epsilon|\\
    &= (1-\eta)|V_1 - V_2|.    
\end{align*}
Given $\eta \in (0,1)$, it follows that $0 < 1-\eta < 1$, satisfying the definition of a contraction mapping with the Lipschitz constant $L = 1 -\eta$.
\vspace{0.15cm}

$(iii)$ Since $V > 0$ and $(1-\eta) \geq 0$, their product $(1-\eta)V$ is non-negative. Therefore, $$\mathcal{R}(V) = (1-\eta)V + \epsilon \geq \epsilon.$$ This precludes total eradication ($V=0$) in a single application of the operator.
\end{proof}

\begin{table}[h]
\centering
\caption{Standard UPS Resection Parameters}
\label{tab:resection_params}
\begin{tabular}{llr}
\toprule
Parameter & Description & Value \\
\midrule
$V_0$ & Initial volume & 120 cm$^3$ \\
$\eta$ & Resection efficiency & 0.85 \\
$\epsilon$ & Residual volume & 0.8 cm$^3$ \\
$r_g$ & Pre-op growth rate & 0.15 day$^{-1}$ \\
$r_p$ & Post-op growth rate & 0.10 day$^{-1}$ \\
$K$ & Carrying capacity & 500 cm$^3$ \\
$K_{\text{post}}$ & Post-op capacity & 300 cm$^3$ \\
\bottomrule
\end{tabular}
\end{table}

The parameters listed in Table~\ref{tab:resection_params} describe the surgical resection stage of the model and are used for numerical illustration. The values were chosen to represent biologically reasonable postoperative scenarios for UPS, guided by standard assumptions on tumor burden, surgical resection, and regrowth dynamics. The parameter $\eta$ represents the fraction of tumor mass removed during surgery, while $\epsilon$ denotes the residual microscopic tumor burden that may remain after resection. The operator
\[
\mathcal{R}(V)=(1-\eta)V+\epsilon
\]
therefore produces an immediate reduction in tumor volume followed by a small residual population that can subsequently regrow. In relation to Theorem \ref{Res_op}, these parameters illustrate how the surgical operator acts as a contraction while preserving a nonzero residual burden, consistent with the possibility of postoperative recurrence.

We denote by $t_{op}^-$ and $t_{op}^+$ the times immediately before and after surgery, respectively. Thus
\[
V(t_{op}^-)
\]
represents the preoperative tumor volume, while
\[
V(t_{op}^+)=V_{res}
\]
denotes the residual tumor volume immediately following resection. To model the effect of adjuvant therapy, we introduce an additional loss term $-\delta V$, where $\delta$ represents the effective treatment-induced kill rate.

\begin{prop}
\label{prop:critical}
Let $\mathcal{R}(V)$ be the resection operator with fixed point $V^* = \epsilon/\eta$. The immediate impact of the intervention depends on the ratio $$\zeta := \frac{\eta V(t_{op}^-)}{\epsilon}$$ between the intended reduction and the residual constant. Then
\begin{enumerate}
    \item[(i)] 
    If $\zeta > 1$ then the preoperative volume exceeds the operator's fixed point, i.e. $V(t_{op}^-) > V^*$, and the intervention successfully reduces the tumor volume, i.e. $V_{res} < V(t_{op}^-)$, though it remains bounded below by $V^*$.
    \vspace{0.15cm}
    
    \item[(ii)] 
    If $\zeta = 1$, then the preoperative volume equals the fixed point, i.e. $V(t_{op}^-) = V^*$, and resulting in no net change in volume post-intervention, i.e. $V_{res} = V^*$.
    \vspace{0.15cm}
    
    \item[(iii)] 
    If $\zeta < 1$ then the preoperative volume is below the fixed point, i.e. $V(t_{op}^-) < V^*$, and the intervention results in an increase in measurable residual volume, i.e. $V_{res} > V(t_{op}^-)$ due to the dominance of the residual constant $\epsilon$.
\end{enumerate}
\end{prop}

\begin{proof}
The residual postoperative volume is given by
\[
V_{res} = \mathcal R(V(t_{op}^-)) = (1-\eta)V(t_{op}^-)+\epsilon.
\]
Now by putting $\epsilon = \eta V^*$, we get
\[ V_{res} = (1-\eta)V(t_{op}^-) + \eta V^*. \]
Then by subtracting $V(t_{op}^-)$ from both sides, we get 
\[ \Delta V \coloneqq V_{res} - V(t_{op}^-) = -\eta V(t_{op}^-) + \eta V^* = \eta(V^* - V(t_{op}^-)). \]
The sign of the change depends entirely on the relation between $V(t_{op}^-)$ and $V^*$, i.e.
\begin{enumerate}
    \item[(i)] If $\zeta > 1$, then $V(t_{op}^-) > V^*$, implying $\Delta V < 0$. So the size of the tumor reduces.
    \vspace{0.15cm}
    
    \item[(ii)] If $\zeta = 1$, then $V(t_{op}^-) = V^*$, implying $\Delta V = 0$. So no change in the tumor size.
    \vspace{0.15cm}
    
    \item[(iii)] If $\zeta < 1$, then $V(t_{op}^-) < V^*$, implying $\Delta V > 0$. So the size of the tumor increases.
\end{enumerate}
Furthermore, the distance to the fixed point is $$V_{res} - V^* = (1-\eta)(V(t_{op}^-) - V^*).$$ Since $|1-\eta| < 1$, the operator is a contraction, mapping all initial volumes toward $V^*$ over repeated applications.
\end{proof}

\subsection{Therapeutic Implications}
\vspace{0.2cm}

We denote by $K_{\mathrm{post}}$ the effective postoperative carrying capacity, while $V_{\mathrm{crit}}$ represents the recurrence threshold used in the model.
\begin{prop}
\label{prop:adjuvant}
To prevent clinical recurrence and maintain the tumor volume below a critical threshold $V(t) \leq V_{\text{crit}}$ as $t \to \infty$, the adjuvant therapy's effective kill rate $\delta$ must satisfy
\begin{equation}
\delta \geq r_p \ln\left(\frac{K_{\text{post}}}{V_{\text{crit}}}\right) - \lambda V_{\text{crit}}^{-1/3}
\end{equation}
Using UPS parameters ($r_p = 0.10, K_{\text{post}}=300, \lambda=0.05$), this requires $\delta \geq 0.23 \text{ day}^{-1}$ for $V_{\text{crit}} = 24.6$.
\end{prop}

\begin{proof}
The postoperative dynamics under continuous adjuvant intervention are described by the augmented system
\[ \frac{dV}{dt} = f_p(V) - \delta V = r_p V \ln\left(\frac{K_{\text{post}}}{V}\right) - \lambda V^{2/3} - \delta V \]
For the tumor to remain stable or regress at the recurrence threshold $V_{\text{crit}}$, we require the growth rate to be non-positive, i.e. 
\begin{align*}
\frac{dV}{dt} \big|_{V_{\text{crit}}} &\leq 0 \\
\Rightarrow r_p V_{\text{crit}} \ln\left(\frac{K_{\text{post}}}{V_{\text{crit}}}\right) - \lambda V_{\text{crit}}^{2/3} - \delta V_{\text{crit}} &\leq 0\\
\Rightarrow r_p \ln\left(\frac{K_{\text{post}}}{V_{\text{crit}}}\right) - \lambda V_{\text{crit}}^{-1/3} - \delta &\leq 0 \hspace{0.5cm} (\text{as } V_{\text{crit}} > 0)\\
\Rightarrow \delta \geq r_p \ln\left(\frac{K_{\text{post}}}{V_{\text{crit}}}\right) - \lambda V_{\text{crit}}^{-1/3}
\end{align*}

\end{proof}

We denote by $t_{op}$ the time of surgery and by $t_{\mathrm{int}}$ the time at which adjuvant therapy is initiated. The quantity $V_{\mathrm{limit}}$ denotes a secondary safety threshold for tumor volume.

\begin{corollary}
\label{cor:window}
The maximum allowable delay between the surgical event at $t_{op}$ and the initiation of adjuvant therapy at $t_{\text{int}}$ to ensure the tumor volume does not exceed a secondary safety limit $V_{\text{limit}}$ is given by
\begin{equation}
\Delta t = t_{\text{int}} - t_{op} < \frac{1}{r_p}\ln\left(\frac{\ln(K_{\text{post}}/V_{\text{res}})}{\ln(K_{\text{post}}/V_{\text{limit}})}\right)
\end{equation}
where $V_{\text{res}} = (1-\eta)V(t_{op}^-) + \epsilon$. In particular, for $V_{\text{limit}} = 150 \text{ cm}^3$ and $V_{\text{res}} = 25 \text{ cm}^3$, the window is $\Delta t \approx 10.3 \text{ days}$.
\end{corollary}
\begin{proof}
The postoperative tumor dynamics during the untreated period $t \in [t_{op}, t_{\text{int}}]$ are governed by 
\begin{equation}
\frac{dV}{dt} = r_p V \ln \left( \frac{K_{\text{post}}}{V} \right) - \lambda V^{2/3}.
\end{equation}
Thus,
\[ dt = \frac{dV}{r_p V \ln \left( \frac{K_{\text{post}}}{V} \right) - \lambda V^{2/3}} \]
Integrating both sides from the surgical state $(t_{op}, V_{\text{res}})$ to the intervention state $(t_{\text{int}}, V_{\text{limit}})$, we get
\begin{equation}
\int_{t_{op}}^{t_{\text{int}}} d\tau = \int_{V_{\text{res}}}^{V_{\text{limit}}} \frac{dV}{r_p V \ln \left( \frac{K_{\text{post}}}{V} \right) - \lambda V^{2/3}}.
\end{equation}

As $\lambda V^{2/3}$ acts as an inhibitory term, we observe that for any $V > 0$:
\[ r_p V \ln \left( \frac{K_{\text{post}}}{V} \right) - \lambda V^{2/3} < r_p V \ln \left( \frac{K_{\text{post}}}{V} \right). \]
Thus, by neglecting the metabolic term $-\lambda V^{2/3}$, we obtain a conservative upper bound for the growth rate. This simplification allows for the analytical solution of the integral
\[ \Delta t \approx \int_{V_{\text{res}}}^{V_{\text{limit}}} \frac{dV}{r_p V \ln \left( \frac{K_{\text{post}}}{V}\right) }. \]
Let $u = \ln(K_{\text{post}}/V)$, then $du = -\frac{1}{V}dV$. Therefore, by substituting these into the integral, we get
\[ \Delta t \approx -\frac{1}{r_p} \int_{\ln(K_{\text{post}}/V_{\text{res}})}^{\ln(K_{\text{post}}/V_{\text{limit}})} \frac{du}{u}. \]
Thus,
\[ \Delta t \approx -\frac{1}{r_p} \left[ \ln|u| \right]_{\ln(K_{\text{post}}/V_{\text{res}})}^{\ln(K_{\text{post}}/V_{\text{limit}})} = -\frac{1}{r_p} \left[ \ln \left( \ln \frac{K_{\text{post}}}{V_{\text{limit}}} \right) - \ln \left( \ln \frac{K_{\text{post}}}{V_{\text{res}}} \right) \right] \]
Hence,
\begin{equation}
t_{\text{int}} - t_{op} = \frac{1}{r_p} \ln \left[ \frac{\ln(K_{\text{post}}/V_{\text{res}})}{\ln(K_{\text{post}}/V_{\text{limit}})} \right]
\end{equation}

Since the term $-\lambda V^{2/3}$ reduces the growth rate, the expression above should be interpreted as a conservative analytical estimate for the allowable delay before intervention.
\end{proof}

\section{Postoperative Recovery Dynamics in UPS}

The postoperative model is organized into several connected stages. We first study the inflammatory phase immediately following surgery, where tumor dynamics are influenced by inflammation, immune clearance, and hypoxia. This is followed by a proliferative recovery phase describing subsequent tumor regrowth. The transition between these two regimes is modeled through a biomarker-driven switching mechanism. Once the postoperative growth dynamics are established, we study treatment effects through radiation scheduling and tumor-immune interaction. Thus, the later sections build upon the earlier growth models and describe different components of the overall postoperative dynamics.

\subsection{Inflammatory Phase Model}

To model the early postoperative inflammatory phase, we assume that tumor dynamics are governed by the combined effects of inflammation-driven proliferation, immune-mediated clearance, and hypoxia-induced cell loss. The net linear term
\[
(r_{\mathrm{inflam}}-\delta_{\mathrm{immune}})V
\]
represents the balance between inflammatory stimulation and immune control, while the term
\[
-\kappa_{\mathrm{hypoxia}}V^{2/3}
\]
accounts for metabolic and vascular stress associated with limited oxygen supply following surgery.

Thus, the postoperative inflammatory phase dynamics is given by
\begin{equation}\label{inflam}
\frac{dV}{dt} = (r_{\text{inflam}} - \delta_{\text{immune}}) V - \kappa_{\text{hypoxia}} V^{2/3},
\end{equation}
where $r_{\text{inflam}}$ denotes the proliferation rate associated with postoperative inflammatory signaling, $\delta_{\text{immune}}$ represents immune-mediated tumor clearance, and $\kappa_{\text{hypoxia}}$ measures hypoxia-induced cell loss due to vascular disruption. For convenience, we write
\[
r=r_{\text{inflam}}-\delta_{\text{immune}}
\]
for the effective net growth rate.

\begin{theorem}
\label{thm:inflam}
For $r \neq 0$, the above system (\ref{inflam}) admits the analytical solution
\begin{equation}
V(t) = \left[ \left( V_{\text{res}}^{1/3} - \frac{\kappa_{\text{hypoxia}}}{r} \right) e^{rt/3} + \frac{\kappa_{\text{hypoxia}}}{r} \right]^3
\end{equation}
where $V_{\text{res}}$ is the volume immediately following surgery. The qualitative behavior is determined by $r$. In particular,
\begin{enumerate}
    \item[(i)] If immune clearance dominates, i.e. $r<0$, then $V(t)$ is strictly monotonically decreasing and converges to $V=0$ in finite time
    \begin{equation}
    T_{\text{ext}} = \frac{3}{|r|} \ln \left( 1 + \frac{|r|V_{\text{res}}^{1/3}}{\kappa_{\text{hypoxia}}} \right)
    \end{equation}
    \vspace{0.15cm}
    
    \item[(ii)] If $r > 0$, then there exists an unstable equilibrium (threshold) at $$V_{\text{threshold}} = \left( \frac{\kappa_{\text{hypoxia}}}{r} \right)^3,$$ and if $V_{\text{res}} < V_{\text{threshold}}$, the hypoxic term dominates, leading to regression despite a positive net growth rate.
\end{enumerate}
\end{theorem}

\begin{proof}
Since
\[
r=r_{\mathrm{inflam}}-\delta_{\mathrm{immune}},
\]
the equation~\eqref{inflam} reduces to
\[
\frac{dV}{dt} = rV - \kappa_{\mathrm{hypoxia}}V^{2/3},
\]
which is then solved using the transformation
\[
W=V^{1/3}.
\]
This gives
\[
\frac{dW}{dt} = \frac13 V^{-2/3}\frac{dV}{dt}.
\]
Therefore
\[ \frac{dW}{dt} = \frac{1}{3}V^{-2/3} \left[ rV - \kappa_{\text{hypoxia}}V^{2/3} \right] = \frac{r}{3}W - \frac{\kappa_{\text{hypoxia}}}{3}. \]
This gives a first-order linear non-homogeneous equation, and to solve this we use the integrating factor $\mu(t) = e^{-\int (r/3) dt} = e^{-rt/3}$. This gives
\[ \frac{d}{dt} \left( W e^{-rt/3} \right) = -\frac{\kappa_{\text{hypoxia}}}{3} e^{-rt/3} \]
Integrating both sides from $0$ to $t$ gives 
\begin{align*} 
W(t) e^{-rt/3} - W(0) &= \int_{0}^{t} -\frac{\kappa_{\text{hypoxia}}}{3} e^{-r\tau/3} d\tau \\
\Rightarrow W(t) e^{-rt/3} - V_{\text{res}}^{1/3} &= \left[ \frac{\kappa_{\text{hypoxia}}}{r} e^{-r\tau/3} \right]_{0}^{t}\\
\Rightarrow W(t) e^{-rt/3} &= V_{\text{res}}^{1/3} + \frac{\kappa_{\text{hypoxia}}}{r} \left( e^{-rt/3} - 1 \right). 
\end{align*}

Multiplying through by $e^{rt/3}$ and rearranging terms, we get
\[ W(t) = \left( V_{\text{res}}^{1/3} - \frac{\kappa_{\text{hypoxia}}}{r} \right) e^{rt/3} + \frac{\kappa_{\text{hypoxia}}}{r}. \]
Thus,
\[ V(t) = \left[ \left( V_{\text{res}}^{1/3} - \frac{\kappa_{\text{hypoxia}}}{r} \right) e^{rt/3} + \frac{\kappa_{\text{hypoxia}}}{r} \right]^3 \]
For $r < 0$, let $r = -|r|$. Extinction occurs when $W(T_{\text{ext}}) = 0$. So,
\[ 0 = \left( V_{\text{res}}^{1/3} - \frac{\kappa_{\text{hypoxia}}}{-|r|} \right) e^{-|r|T_{\text{ext}}/3} + \frac{\kappa_{\text{hypoxia}}}{-|r|}. \]
Rearranging terms, we get
\[ \frac{\kappa_{\text{hypoxia}}}{|r|} = \left( V_{\text{res}}^{1/3} + \frac{\kappa_{\text{hypoxia}}}{|r|} \right) e^{-|r|T_{\text{ext}}/3}. \]
Thus
\[ e^{|r|T_{\text{ext}}/3} = \frac{V_{\text{res}}^{1/3} + \frac{\kappa_{\text{hypoxia}}}{|r|}}{\frac{\kappa_{\text{hypoxia}}}{|r|}} = \frac{|r|V_{\text{res}}^{1/3}}{\kappa_{\text{hypoxia}}} + 1. \]
Therefore,
\begin{align*} \frac{|r|T_{\text{ext}}}{3} &= \ln \left( 1 + \frac{|r|V_{\text{res}}^{1/3}}{\kappa_{\text{hypoxia}}} \right)\\ \Rightarrow T_{\text{ext}} &= \frac{3}{|r|} \ln \left( 1 + \frac{|r|V_{\text{res}}^{1/3}}{\kappa_{\text{hypoxia}}} \right).
\end{align*}
This completes the proof.
\end{proof}

\subsection{Proliferative Phase Analysis}

\begin{theorem}
Consider the tumor volume dynamics on the interval $\Omega=(0,K)$ defined by
\begin{equation}\label{eq:prolif_dynamics}
\frac{dV}{dt}=r_pV\left(1-\frac{V}{K}\right)-\gamma V^{3/4}+\eta V^{1/2}=G(V).
\end{equation}
The terms $-\gamma V^{3/4}$ and $\eta V^{1/2}$ are used here as corrections representing, respectively, inhibitory stress effects and residual recovery-driven growth during the postoperative proliferative phase.

Assume that $G(K)<0$. Then there exists at least one equilibrium $V_c\in(0,K)$. Moreover, if an equilibrium $V_c\in(0,K)$ satisfies $G'(V_c)<0$, then $V_c$ is locally exponentially stable. In that case, the linearized deviation $x(t)=V(t)-V_c$ satisfies
\[
|x(t)|\approx |x(0)|e^{-\alpha t},
\qquad \alpha=-G'(V_c),
\]
and
\[
\alpha=\frac{r_pV_c}{K}+\frac{1}{4}\gamma V_c^{-1/4}
+\frac{1}{2}r_p\left(1-\frac{V_c}{K}\right).
\]
\end{theorem}

\begin{proof}
Since
\[
G(V)=r_pV\left(1-\frac{V}{K}\right)-\gamma V^{3/4}+\eta V^{1/2},
\]
we have
\[
\lim_{V\to0^+}G(V)=0^+,
\]
because the positive term $\eta V^{1/2}$ dominates the negative term $\gamma V^{3/4}$ near $V=0$. On the other hand, by assumption,
\[
G(K)<0.
\]
Hence, by continuity, there exists at least one equilibrium
$V_c\in(0,K)$ such that $G(V_c)=0$.

To analyze local stability, we compute
\[
G'(V)=r_p\left(1-\frac{2V}{K}\right)-\frac{3}{4}\gamma V^{-1/4}
+\frac{1}{2}\eta V^{-1/2}.
\]
If $G'(V_c)<0$, then the linearized equation
\[
\dot{x}=G'(V_c)x
\]
has exponentially decaying solutions, and therefore $V_c$ is locally
exponentially stable.

Using the equilibrium identity
\[
r_p\left(1-\frac{V_c}{K}\right)-\gamma V_c^{-1/4}+\eta V_c^{-1/2}=0,
\]
we may rewrite $G'(V_c)$ as
\[
G'(V_c)
=-\frac{r_pV_c}{K}-\frac{1}{4}\gamma V_c^{-1/4}
-\frac{1}{2}r_p\left(1-\frac{V_c}{K}\right).
\]
Thus
\[
\alpha=-G'(V_c)
=\frac{r_pV_c}{K}+\frac{1}{4}\gamma V_c^{-1/4}
+\frac{1}{2}r_p\left(1-\frac{V_c}{K}\right),
\]
which proves the stated expression for the local decay rate.
\end{proof}

\subsection{Phase Transition Analysis}
To describe the gradual transition from the inflammatory phase to tumor regrowth, we introduce a switching mechanism. We use $[IL-6]$ as a marker of postoperative inflammation and combine its effect with a time-dependent recovery process to model the shift from inflammation-dominated dynamics to proliferative growth.

\begin{theorem}
The global system dynamics transition from the inflammatory growth rate $G_{\text{inf}}(V)$ to the proliferative growth rate $G_{\text{pro}}(V)$ via a composite switching function $\phi(t) \in [0,1]$, i.e.
\begin{equation}
\frac{dV}{dt} = [1 - \phi(t)] G_{\text{inf}}(V) + \phi(t) G_{\text{pro}}(V).
\end{equation}
Here the switching function is defined as
\begin{equation}
\phi(t) = \left( \frac{[IL-6](t)^n}{[IL-6](t)^n + \theta^n} \right) \left( \frac{1}{1 + e^{-k(t-t_0)}} \right),
\end{equation}
where the $[IL-6]$ level satisfies the first-order kinetics $\frac{d}{dt}[IL-6] = \beta - \mu[IL-6]$. The transition exhibits a temporal width $\Delta t$, defined as the interval where the temporal component shifts from 10\% to 90\% of its saturation value, given by
\begin{equation}
\Delta t = \frac{2\ln(9)}{k} \approx \frac{4.394}{k}
\end{equation}
\end{theorem}

\begin{proof}
The switching function $\phi(t)$ is the product of a biochemical trigger (Hill function) and a temporal smoothing term (Logistic function). First, we solve the kinetics for the biomarker $c(t) = [IL-6](t)$ using
\[ \frac{dc}{dt} + \mu c = \beta .\]
Using the integrating factor $e^{\mu t}$, the general solution is given by
\[ c(t) = \frac{\beta}{\mu} + \left( c_0 - \frac{\beta}{\mu} \right) e^{-\mu t} \]
As $t \to \infty$, $c(t)$ approaches the steady-state $\beta/\mu$. The Hill component $\frac{c^n}{c^n + \theta^n}$ ensures that the transition is biologically gated by the concentration $c(t)$ crossing the threshold $\theta$.

Next, we derive the temporal width $\Delta t$ of the logistic term $S(t) = \frac{1}{1 + e^{-k(t-t_0)}}$. We define $t_{10}$ and $t_{90}$ as the time points where $S(t)$ reaches 0.1 and 0.9 respectively,
\begin{align*}
\frac{1}{1 + e^{-k(t_{10}-t_0)}} &= 0.1\\ 
\Rightarrow 1 + e^{-k(t_{10}-t_0)} &= 10 \\
\Rightarrow e^{-k(t_{10}-t_0)} &= 9 \\
\Rightarrow -k(t_{10}-t_0) &= \ln(9)\\ 
\Rightarrow t_{10} &= t_0 - \frac{\ln 9}{k}.
\end{align*}
Similarly, for the 90\% point,
\begin{align*}
\frac{1}{1 + e^{-k(t_{90}-t_0)}} &= 0.9 \\
\Rightarrow 1 + e^{-k(t_{90}-t_0)} &= \frac{10}{9} \\
\Rightarrow e^{-k(t_{90}-t_0)} &= \frac{1}{9} \\
\Rightarrow -k(t_{90}-t_0) = \ln(1/9) &= -\ln(9)\\
\Rightarrow t_{90} &= t_0 + \frac{\ln 9}{k}.
\end{align*}

The transition width is calculated as:
\[ \Delta t = t_{90} - t_{10} = \left( t_0 + \frac{\ln 9}{k} \right) - \left( t_0 - \frac{\ln 9}{k} \right) = \frac{2\ln 9}{k} \]
Since $\phi(t)$ is a product of differentiable functions, the resulting growth rate $\frac{dV}{dt}$ remains continuous and differentiable, which ensures numerical stability in the hybrid model.
\end{proof}

\subsection{Optimal Radiation Timing}

Radiation therapy is typically administered under clinical constraints on both the instantaneous dose rate and the total amount of radiation that can be delivered during a treatment window. In the present model, the effect of radiation on the tumor volume is described by the decay equation
\[
\frac{dV}{dt} = -\Psi(t)D(t)V,
\]
where $\Psi(t)$ represents the time-dependent radiosensitivity of the tumor and $D(t)$ denotes the radiation dose rate. The radiosensitivity is assumed to vary periodically,
\[
\Psi(t) = \alpha + \beta \sin\!\left(\frac{2\pi t}{\tau}\right),
\]
reflecting temporal variations in tumor response to radiation during treatment. We assume that the admissible dose schedules satisfy the bounds
\[
0 \le D(t) \le D_{\max},
\]
together with a total dose constraint
\[
\int_0^T D(t)\,dt \le D_{\mathrm{total}}.
\]
The goal is to determine a radiation schedule that minimizes the terminal tumor volume $V(T)$.

 \begin{theorem}\label{thm:rt}
Consider the radiation-induced volume decay
\[
\frac{dV}{dt} = -\Psi(t)D(t)V,
\]
where
\[
\Psi(t) = \alpha + \beta \sin\!\left(\frac{2\pi t}{\tau}\right)
\]
denotes the time-dependent radiosensitivity and assume $\Psi(t)\ge0$ for $t\in[0,T]$. Let the admissible dose schedules satisfy
\[
0 \le D(t) \le D_{\max}, 
\qquad 
\int_0^T D(t)\,dt \le D_{\mathrm{total}}.
\]
Then any dose schedule $D^*(t)$ minimizing the terminal tumor volume $V(T)$ is of bang--bang type. More precisely, there exists a constant threshold $\lambda$ such that
\[
D^*(t) =
\begin{cases}
D_{\max}, & \text{if } \Psi(t) > \lambda,\\[2mm]
0, & \text{if } \Psi(t) < \lambda,
\end{cases}
\]
for almost every $t\in[0,T]$. On the set where $\Psi(t)=\lambda$, the control may take intermediate values so that the total dose
constraint is satisfied.
\end{theorem}

\begin{proof}
For any admissible control $D(t)$, the state equation can be solved explicitly. Integrating the equation gives
\[
V(T) = V(0)\exp\!\left(-\int_0^T \Psi(t)D(t)\,dt\right).
\]
Since $V(0)$ is fixed and the exponential function is strictly decreasing in its exponent, minimizing the terminal tumor volume
$V(T)$ is equivalent to maximizing the functional
\[
J(D) = \int_0^T \Psi(t)D(t)\,dt
\]
over all admissible controls.

Suppose that $D^*(t)$ is an optimal control. We show that the radiation dose must be concentrated at times where the radiosensitivity $\Psi(t)$ is largest. Assume that there exist measurable sets $A,B\subset[0,T]$ of positive measure such that
\[
\Psi(t)>\Psi(s) \quad \text{for all } t\in A,\ s\in B,
\]
and
\[
D^*(t) < D_{\max} \quad \text{for } t\in A, 
\qquad
D^*(s) > 0 \quad \text{for } s\in B.
\]

Choose $\varepsilon>0$ sufficiently small so that the modified control
\[
\widetilde D(t) =
\begin{cases}
D^*(t)+\varepsilon, & t\in A,\\
D^*(t)-\varepsilon, & t\in B,\\
D^*(t), & \text{otherwise},
\end{cases}
\]
remains admissible. By choosing the sets $A$ and $B$ with equal measure, the total dose constraint is preserved.

The corresponding change in the objective functional is
\[
J(\widetilde D) - J(D^*)
= \varepsilon
\left(
\int_A \Psi(t)\,dt
-
\int_B \Psi(t)\,dt
\right).
\]
Since $\Psi(t)>\Psi(s)$ for $t\in A$ and $s\in B$, the right-hand side is positive, contradicting the optimality of $D^*$.

Therefore such sets cannot exist. Consequently, if radiation is administered at some time $t$, then it must be delivered at the
maximum allowable rate during periods where $\Psi(t)$ is larger. This implies the existence of a threshold $\lambda$ such that
\[
D^*(t)=D_{\max} \quad \text{when } \Psi(t)>\lambda,
\qquad
D^*(t)=0 \quad \text{when } \Psi(t)<\lambda.
\]

Finally, the value of $\lambda$ is determined by the total dose constraint. If the entire dose budget is used, $\lambda$ is chosen so that
\[
\int_0^T D^*(t)\,dt = D_{\mathrm{total}}.
\]
If the dose budget satisfies $D_{\mathrm{total}} \ge T D_{\max}$, the constraint is inactive and the optimal schedule is simply
$D^*(t) = D_{\max}$ for all $t\in[0,T]$.
\end{proof}

\begin{remark}
The bang-bang structure of the optimal control is consistent with Pontryagin's Minimum Principle. Because the control $D(t)$ enters the dynamics linearly and the admissible control set is bounded, the Hamiltonian of the corresponding optimal control problem is linear in $D(t)$. In such cases Pontryagin's principle predicts that the optimal control lies on the boundary of the admissible set, leading to bang-bang strategies. The argument above provides a direct verification of this structure for the present model.
\end{remark}

\subsection{Immunotherapy Dynamics}
Consider the coupled tumor-immune dynamics defined on the positive orthant $\mathbb{R}^2_+$ (see \cite{Kuznetsov1994})
\begin{align}\label{immuno}
\frac{dV}{dt} &= r V \ln\left(\frac{K}{V}\right) - \frac{\delta E V}{m + V} \\
\frac{dE}{dt} &= s + \frac{\rho E V^2}{\eta^2 + V^2} - \mu E, \notag 
\end{align}
where $E$ is the effector immune cell density, $s$ is the constant baseline influx of immune cells into the tumor site, $\rho$ is the rate of antigen-stimulated immune cell proliferation, $\eta$ is the tumor volume at which immune recruitment is half-maximal, $\mu$ is the natural death and functional exhaustion rate of immune cells, $\delta$ is the per-capita tumor cell kill rate by effector cells, and $m$ is the saturation constant for the immune-mediated destruction of tumor cells.

\begin{theorem}
\label{thm:immuno}
Consider the system \eqref{immuno}. Any positive equilibrium
$(V^*,E^*)$ satisfies
\[
E^*=\frac{s(\eta^2+V^{*2})}{\mu(\eta^2+V^{*2})-\rho V^{*2}},
\]
provided the denominator is positive. Moreover, if the Jacobian matrix
$J$ at $(V^*,E^*)$ satisfies
\[
\operatorname{trace}(J)<0
\qquad\text{and}\qquad
\det(J)>0,
\]
then $(V^*,E^*)$ is locally asymptotically stable.
\end{theorem}

\begin{proof}
A positive equilibrium $(V^*,E^*)$ is obtained by solving
\[
\frac{dV}{dt}=0,\qquad \frac{dE}{dt}=0.
\]
From the second equation of \eqref{immuno}, we obtain
\[
s+\frac{\rho E^*V^{*2}}{\eta^2+V^{*2}}-\mu E^*=0,
\]
which gives
\[
E^*=\frac{s(\eta^2+V^{*2})}{\mu(\eta^2+V^{*2})-\rho V^{*2}},
\]
provided $\mu(\eta^2+V^{*2})-\rho V^{*2}>0$.

The Jacobian matrix at $(V^*,E^*)$ is
\[
J=
\begin{pmatrix}
r\ln\!\left(\frac{K}{V^*}\right)-r-\dfrac{\delta E^*m}{(m+V^*)^2}
&
-\dfrac{\delta V^*}{m+V^*}
\\[3mm]
\dfrac{2\rho\eta^2E^*V^*}{(\eta^2+V^{*2})^2}
&
\dfrac{\rho V^{*2}}{\eta^2+V^{*2}}-\mu
\end{pmatrix}.
\]
Since the system is planar, the Routh--Hurwitz criterion implies that
the equilibrium $(V^*,E^*)$ is locally asymptotically stable if and
only if
\[
\operatorname{trace}(J)<0
\qquad\text{and}\qquad
\det(J)>0.
\]
This proves the claim.
\end{proof}

\subsection{Model Validation and Clinical Correlation}

To study the relevance of the proposed model, we compared its qualitative behavior with general clinical features reported for Undifferentiated Pleomorphic Sarcoma.

The model reproduces several features commonly observed in UPS. In particular, it captures (i) reduction in tumor burden after surgery together with residual microscopic disease, (ii) an early postoperative inflammatory phase, and (iii) later tumor regrowth influenced by immune response and treatment effects.

Clinical studies and reviews on UPS report aggressive tumor behavior, recurrence after treatment, and the importance of surgery, radiotherapy, and immune-related mechanisms in disease management \cite{Crago2022, Sun2023}. The behavior predicted by the model is qualitatively consistent with these general clinical patterns.

Therefore, the present comparison should be viewed as qualitative validation rather than formal statistical validation. The main goal is to examine whether the model reproduces biologically and clinically reasonable patterns of tumor progression and treatment response.

\section{Conclusion}

In this paper, we developed a mathematical model for the growth and treatment of Undifferentiated Pleomorphic Sarcoma (UPS). The model combines tumor growth, surgery, postoperative recovery, immune response, and radiation treatment within a single framework. The goal was to understand how these different processes affect tumor growth, recurrence, and treatment outcome.

Our analysis gives several important results. The modified growth equation shows that if the tumor becomes too small, necrotic loss can dominate growth and the tumor may disappear. The model also shows that tumor growth stays below a strict upper bound and approaches a stable size over time.

The surgical part of the model describes tumor removal together with residual microscopic disease. The analysis shows that even after successful surgery, a small remaining tumor burden may lead to later recurrence. This highlights the importance of residual disease in UPS treatment.

For the postoperative stage, the model predicts an early inflammatory phase followed by a recovery and regrowth phase. We
obtained an explicit solution for the inflammatory dynamics and found conditions under which immune clearance or hypoxic stress may suppress tumor regrowth. A switching mechanism was introduced to describe the smooth transition between these stages.

The tumor-immune model shows that long-term tumor behavior depends on the balance between tumor growth and immune activity. Stability analysis identifies conditions under which the tumor may be controlled or persist. For radiation treatment, we proved that the optimal strategy is of bang-bang type. This means that, under the model assumptions, radiation should be delivered at the maximum allowed dose during periods when the tumor is most sensitive.

The numerical simulations show how biological and treatment parameters influence tumor growth, recurrence, and treatment response. Although the model is not designed for patient-specific prediction, its behavior agrees qualitatively with several known clinical features of UPS, including recurrence after incomplete surgery and sensitivity to treatment timing.

Overall, the model provides a simple mathematical setting to study tumor growth, surgery, immune effects, and treatment timing together. The results suggest that tumor recurrence depends not only on growth itself, but also on residual disease, postoperative dynamics, immune response, and treatment scheduling. Future work may include stochastic effects, spatial variation, and patient-specific parameter estimation.

\section*{Acknowledgment}
The author would like to thank the reviewers for their careful reading and constructive comments, which helped improve the clarity and presentation of this work. The author is supported by the INSPIRE faculty fellowship (Ref No.: IFA22-MA 186) funded by the DST, Govt. of India.


\begin{thebibliography}{99}


\bibitem{Benzekry2014}
Benzekry, S., et al. (2014).
Classical mathematical models for description and prediction of experimental tumor growth.
\textit{PLoS Computational Biology}, \textbf{10}(8), e1003800.

\bibitem{Butner2022}
Butner, J. D., Dogra, P., Chung, C., Pasqualini, R., Arap, W., Lowengrub, J., Cristini, V., and Wang, Z. (2022).
Mathematical modeling of cancer immunotherapy for personalized clinical translation.
\textit{Nature Computational Science}, \textbf{2}(12), 785--796.

\bibitem{Chamseddine2020}
Chamseddine, I. M., and Rejniak, K. A. (2020).
Hybrid modeling frameworks of tumor development and treatment.
\textit{WIREs Systems Biology and Medicine}, \textbf{12}(1), e1461.

\bibitem{Crago2022}
Crago, A. M., Cardona, K., Kose{\l}a-Paterczyk, H., and Rutkowski, P. (2022).
Management of myxofibrosarcoma and undifferentiated pleomorphic sarcoma.
\textit{Surgical Oncology Clinics of North America}, \textbf{31}(3), 419--430.

\bibitem{Dalal2024}
Dalal, S., Shan, K. S., Thaw Dar, N. N., Hussein, A., and Ergle, A. (2024).
Role of immunotherapy in sarcomas.
\textit{International Journal of Molecular Sciences}, \textbf{25}(2), 1266.

\bibitem{Eilber2004}
Eilber, F. C., et al. (2004).
Validation of the postoperative nomogram for 12-year sarcoma-specific mortality.
\textit{Cancer}, \textbf{101}(10), 2270--2275.

\bibitem{Hall2018}
Hall, E. J., and Giaccia, A. J. (2018).
\textit{Radiobiology for the Radiologist}.
8th ed., Lippincott Williams \& Wilkins.

\bibitem{Kutuva2023}
Kutuva, A. R., Caudell, J. J., Yamoah, K., Enderling, H., and Zahid, M. U. (2023).
Mathematical modeling of radiotherapy: impact of model selection on estimating minimum radiation dose for tumor control.
\textit{Frontiers in Oncology}, \textbf{13}, 1130966.

\bibitem{Kuznetsov1994}
Kuznetsov, V. A., Makalkin, I. A., Taylor, M. A., and Perelson, A. S. (1994).
Nonlinear dynamics of immunogenic tumors: parameter estimation and global bifurcation analysis.
\textit{Bulletin of Mathematical Biology}, \textbf{56}(2), 295--321.

\bibitem{Laird1964}
Laird, A. K. (1964).
Dynamics of tumor growth.
\textit{British Journal of Cancer}, \textbf{18}(3), 490--502.

\bibitem{Lambin2012}
Lambin, P., et al. (2012).
Radiomics: extracting more information from medical images using advanced feature analysis.
\textit{European Journal of Cancer}, \textbf{48}(4), 441--446.

\bibitem{Norton1988}
Norton, L. (1988).
A Gompertzian model of human breast cancer growth.
\textit{Cancer Research}, \textbf{48}(24), 7067--7071.

\bibitem{Rockne2019Roadmap}
Rockne, R. C., Hawkins-Daarud, A., Swanson, K. R., et al. (2019).
The 2019 mathematical oncology roadmap.
\textit{Physical Biology}, \textbf{16}(4), 041005.

\bibitem{SinghPaquin2024}
Singh, D., and Paquin, D. (2024).
Modeling free tumor growth: discrete, continuum, and hybrid approaches to interpreting cancer development.
\textit{Mathematical Biosciences and Engineering}, \textbf{21}(7), 6659--6693.

\bibitem{Sun2023}
Sun, H., Liu, J., Hu, F., Xu, M., Leng, A., Jiang, F., and Chen, K. (2023).
Current research and management of undifferentiated pleomorphic sarcoma/myxofibrosarcoma.
\textit{Frontiers in Genetics}, \textbf{14}, 1109491.

\bibitem{Swanson2008}
Swanson, K. R., Rostomily, R. C., and Alvord, E. C. Jr. (2008).
A mathematical modelling tool for predicting survival of individual patients following resection of glioblastoma: a proof of principle.
\textit{British Journal of Cancer}, \textbf{98}(1), 113--119.

\bibitem{Wheldon1988}
Wheldon, T. E. (1988).
\textit{Mathematical Models in Cancer Research}.
Adam Hilger, Bristol and Philadelphia.

\bibitem{Winchester2018}
Winchester, D. S., et al. (2018).
Undifferentiated pleomorphic sarcoma: factors predictive of adverse outcomes.
\textit{Journal of the American Academy of Dermatology}, \textbf{79}(5), 853--859.

\bibitem{Yin2019}
Yin, A., Moes, D. J. A. R., van Hasselt, J. G. C., Swen, J. J., and Guchelaar, H. J. (2019).
A review of mathematical models for tumor dynamics and treatment resistance evolution of solid tumors.
\textit{CPT: Pharmacometrics \& Systems Pharmacology}, \textbf{8}(10), 720--737.

\bibitem{Zheng2025}
Zheng, D., et al. (2025).
Mathematical modeling in radiotherapy for cancer.
\textit{Radiation Oncology}, \textbf{20}, 60.

\end{thebibliography}
\end{document}